\documentclass[a4paper,11pt]{article}
\usepackage[dvipdfmx]{graphicx}
\usepackage{amsmath,amsthm,amssymb,amsfonts}
\usepackage{mathrsfs}
\usepackage[dvipdfmx, usenames]{color}
\usepackage{bm}
\usepackage[margin=20mm]{geometry}
\usepackage{braket}
\usepackage{amsmath}
\usepackage{amssymb}
\usepackage{xcolor}
\usepackage[dvipdfmx]{color}
\usepackage[dvipdfmx]{graphicx}
\usepackage{mathtools}
\usepackage{here}

\newtheorem{thm}{Theorem}[section]
\newtheorem*{thm*}{Theorem}

\newtheorem{cor}[thm]{Corollary}
\newtheorem{prop}[thm]{Proposition}

\theoremstyle{definition}

\makeatletter

\usepackage{latexsym}

\theoremstyle{remark}

\numberwithin{equation}{section}

\newcommand{\s}{\textit}

\usepackage[abbrev,nobysame]{amsrefs}

\makeatletter

\def\id#1{\def\@id{#1}}
\def\department#1{\def\@department{#1}}

\def\@maketitle{
\begin{center}
{\LARGE\bf \@title \par}
\vspace{10pt}
{\large \@department \par} 
{\@author \par}
\vspace{10pt}
\end{center}

\thispagestyle{empty}
}

\makeatother

\title{A virtualized skein relation for a multivariable polynomial invariant }
\department{Nagoya City University}
\author{Moemi Hiraki \footnote{e-mail : m.hiraki@nsc.nagoya-cu.ac.jp}}
\begin{document}
\maketitle
\thispagestyle{empty}

\begin{abstract}
The virtual skein relation for the Jones polynomial of the virtual link diagram was introduced by N. Kamada, S. Nakabo, and S. Satoh(\cite{skein}). H. A. Dye, L. H. Kauffman, and Y. Miyazawa introduced multivariable polynomial, an invariant of virtual links, which is a refinement of Jones polynomial(\cite{arrow},\cite{miyazawa}).
In this paper, we give a skein relation for the multivariable polynomials among positive, negative, and virtual crossings with some restrictions. We apply this relation to study some properties of virtual links obtained by replacing a real crossing by a virtual crossing.
\end{abstract}

\vspace{-20pt}

\section{Introduction}
In 1996, L. H. Kauffman introduced virtual knot theory as generalization of classical knot theory(\cite{kauffman}). In \cite{kauffman}, \s{f}-polynomial(Jones polynomial) was defined. It is an extension of Jones polynomial of classical links, denoted by $f_D(A) \in \mathbb{Z}[A^{\pm1}]$.
The following theorem holds in the category of virtual links as well as classical ones. 
\begin{thm}(L.H.Kauffman \cite{kauffman})\\
Let $(D_+,D_-,D_0)$ be a skein triple of oriented virtual link diagram. Then we have\\
\vspace{-10pt}
\begin{equation*}
A^4f_{D_+}(A)-A^{-4}f_{D_-}(A)+(A^2-A^{-2})f_{D_0}(A)=0
\end{equation*}
\end{thm}
Here a skein triple $(D_+,D_-,D_0)$ means a triple of virtual link diagrams such that $D_-$ is obtained from $D_+$ by crossing change at a positive crossing point $p$ and $D_0$ is obtained from $D_+$ by smoothing $p$ as usual. 
A virtual skein triple is a triple $(D_+,D_-,D_v)$ of virtual link diagrams such that $D_-$ is obtained from $D_+$ by crossing change at a positive crossing point $p$ and $D_v$ is obtained from $D_+$ by replacing $p$ with a virtual crossing.
\begin{thm}(N. Kamada, S. Nakabo, S. Satoh \cite{skein})\\
Let $(D_+,D_-,D_v)$ be a virtual skein triple such that $D_+$ is a checkerboard colorable virtual link diagram, Then we have\\
\vspace{-10pt}
\begin{equation*}
A^3f_{D_+}(A)+A^{-3}f_{D_-}(A)=(A^3+A^{-3})f_{D_v}(A)
\end{equation*}
\end{thm}
We give the definition of a checkerboard colorable virtual link diagram in Section \ref{pre}. In \cite{relation}, virtualized skein relations are given under different conditions. 
The multivariable polynomial invariant was defined by H. A. Dye, L. H. Kauffman (\cite{arrow}) and Y. Miyazawa (\cite{miyazawa}), which is a refinement of \s{f}-polynomial. 

A virtual link diagram is called almost classical if it admit an Alexander numbering. Where Alexander numbering is explained in Section2.  
An almost classical virtual link diagram is checkerboard colorable. 

In this paper, we discuss virtualized skein relation for the multivariable polynomial invariont of an almost classical virtual link. 

\section{Preparation}\label{pre}
In this section we recall the definitions of oriented cut points and cut systems, and multivariable polynomial invariants.

\subsection{Virtual links and their Alexander numberings}

A \s{virtual link diagram} is a collection of immersed oriented loops in $\mathbb{R}^2$ such that the multiple points are transverse double points and they are classified into classical crossings and virtual crossings: A \s{classical crossing} is an intersection with over/under information as in usual link diagrams, and a \s{virtual crossing} is an intersection without over/under information (\cite{kauffman}). A virtual crossing is depicted as a crossing encircled with a small circle. A classical crossing is also called a positive or negative as usual in knot  theory.

\s{Generalized Reidemeister moves} are local moves depicted in Fig.\ref{GR}: The 3 moves on top are \s{(classical) Reidemeister moves} and the 4 moves on the bottom are so-called \s{virtual Reidemeister moves}. Two link diagrams $D$ and $D'$ are said to be \s{equivarent} if they are related by a finite sequence of generalized Reidemeister moves and isotopic of $\mathbb{R}^2$.  A \s{virtual link} is an equivalence class of virtual link diagrams. 

\begin{figure}[htbp]
\centering
\includegraphics[scale=0.35]{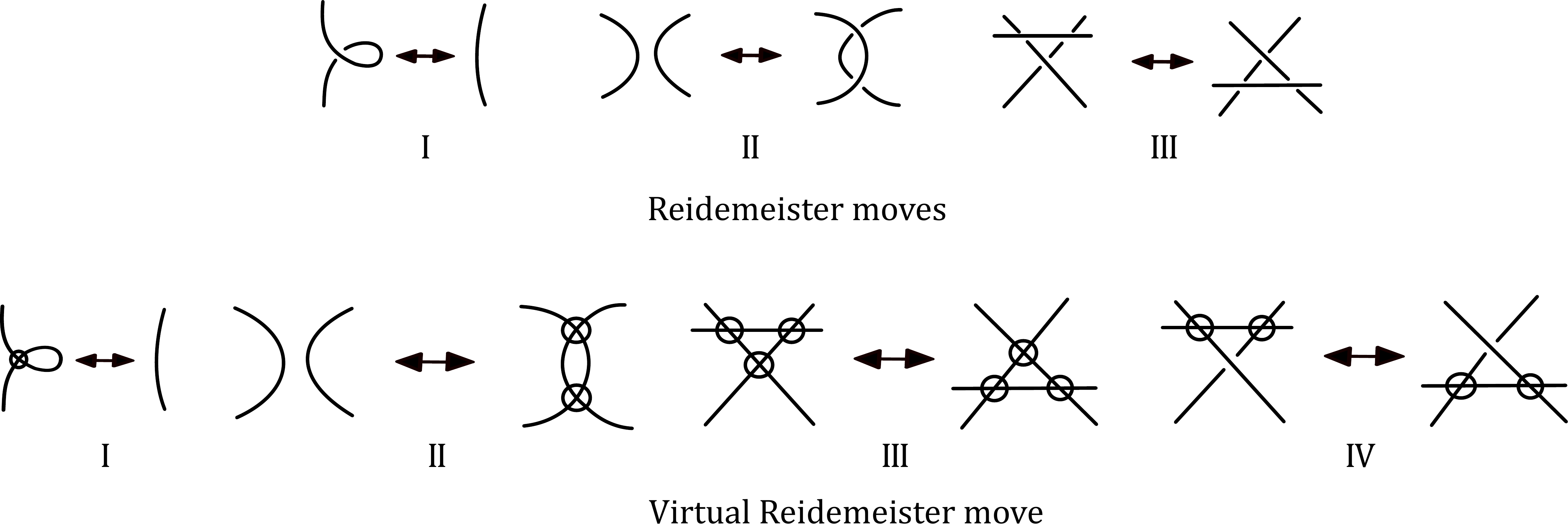}
\caption{Generalized Reidemeister move.} 
\label{GR}
\end{figure}

Let $D$ be a virtual link diagram. A \s{semi-arc} of $D$ is an immersed arc in a  component of $D$ between two classical crossings or an immersed loop missing classical crossings of $D$. An \s{Alexander numbering} of $D$ is an assignment of a number of $\mathbb{Z}$ to each semi-arc of $D$ such that for each classical crossing the numbers of 4 semi-arcs around it are as shown in Fig.\ref{Alexanderc} for some i $\in$ $\mathbb{Z}$.

\begin{figure}[htbp]
\centering
\includegraphics[scale=0.4]{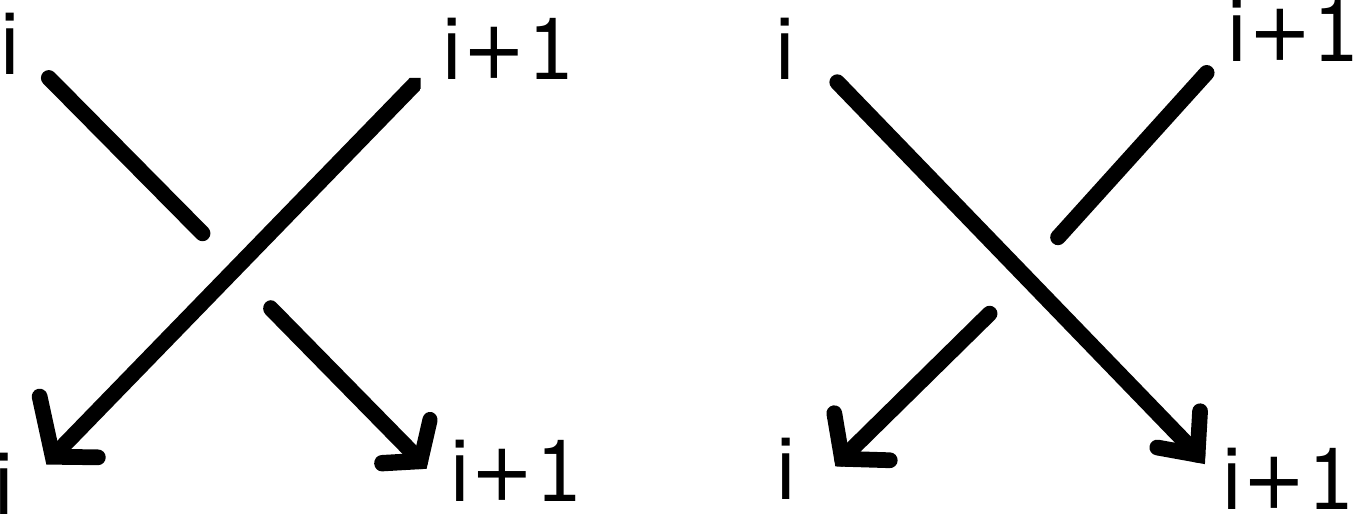}
\caption{Alexander numbering.} 
\label{Alexanderc}
\end{figure}

Note that a virtual crossing is an intersection of two semi-arcs, and the numbers assigned to semi-arcs are as in Fig.\ref{Alexanderv}

\begin{figure}[htbp]
\centering
\includegraphics[scale=0.4]{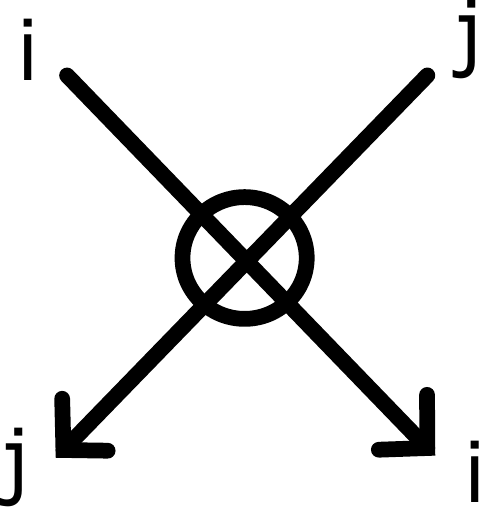}
\caption{Alexander numbering around a virtual crossing.} 
\label{Alexanderv}
\end{figure}

An example of a virtual link diagram with an Alexander numbering is depicted in Fig.\ref{trefoil}. A classical link diagram always admits an Alexander numbering. 

\begin{figure}[H]
\centering
\includegraphics[scale=0.4]{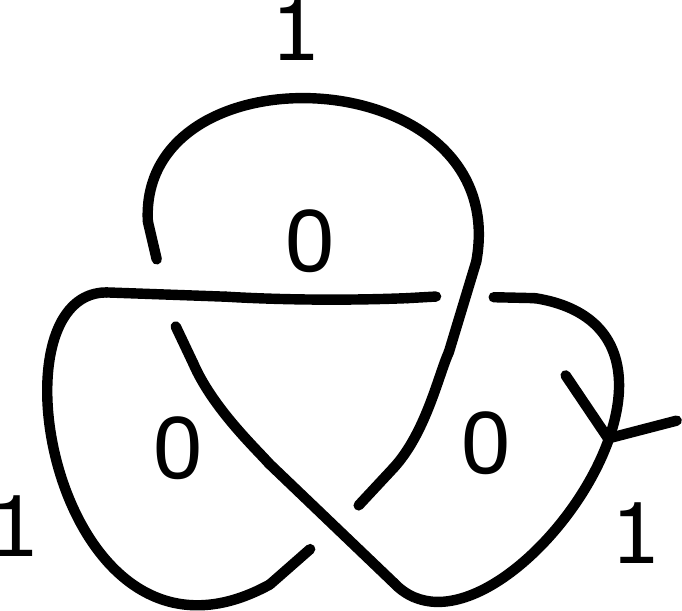}
\caption{Alexander numbering of trefoil.} 
\label{trefoil}
\end{figure}

Not every virtual link diagram admits an Alexander numbering. The virtual link diagram depicted in Fig.\ref{exalmost} does admit Alexander numbering, and the virtual link diagram in Fig.\ref{vtrefoil} does not.
 
\begin{figure}[H]
\centering
\includegraphics[scale=0.2]{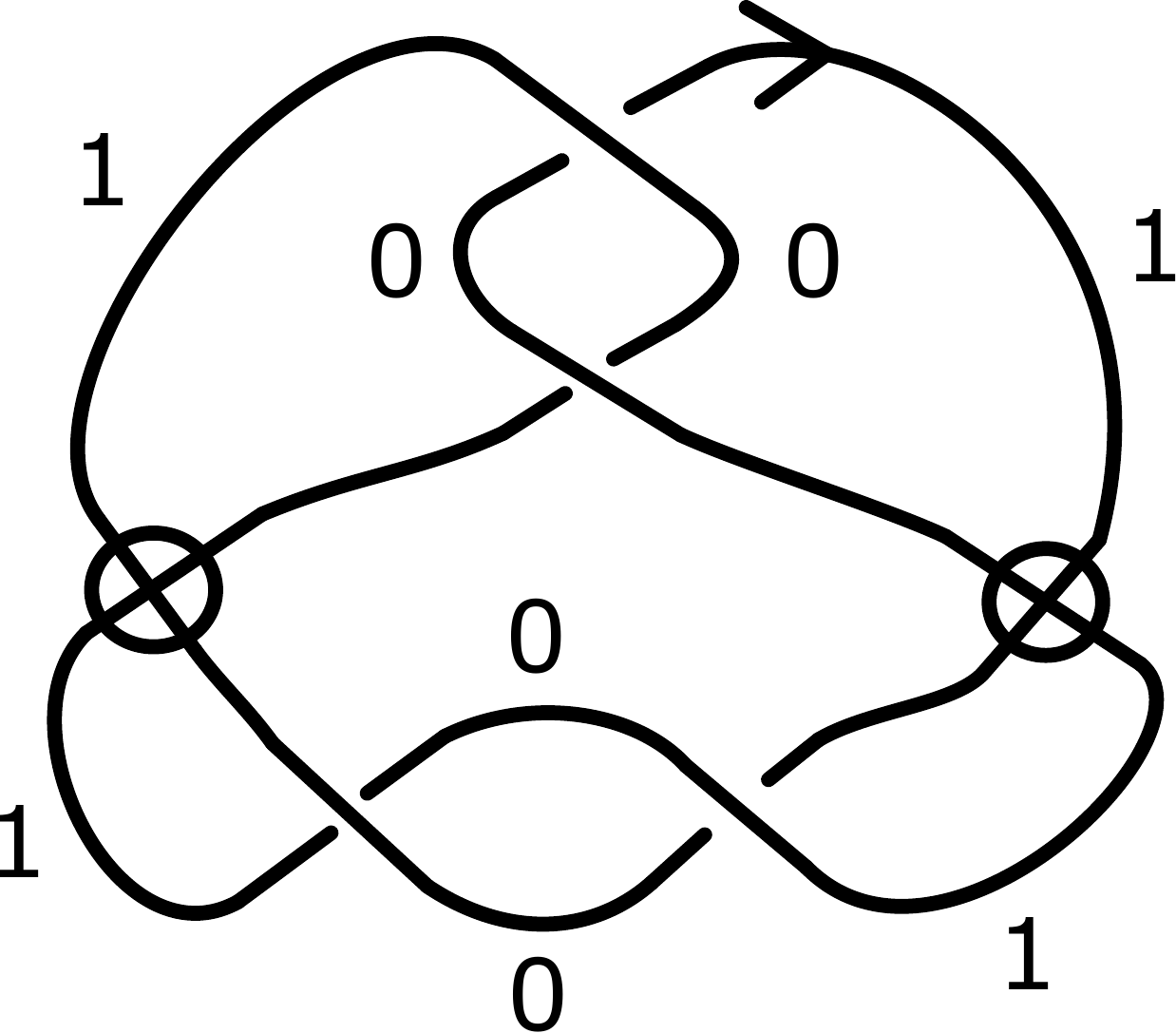}
\caption{Alexander numbering of virtual link diagram.} 
\label{exalmost}
\end{figure}

\begin{figure}[H]
\centering
\includegraphics[scale=0.4]{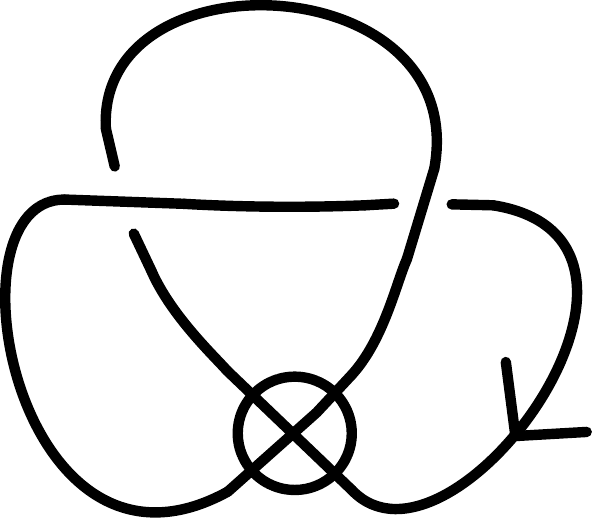}
\caption{Virtual link diagram which does not admit an Alexander numbering.} 
\label{vtrefoil}
\end{figure}

If a virtual link diagram admits an Alexander numbering in $\mathbb{Z}_2$ (=$\mathbb{Z}/2\mathbb{Z}$), it is said to be checkerboard colorable ($\mathbb{Z}_2$ almost classical). Note that the definition of a checkerboard colorable virtual link diagram is equivarent to that of \cite{cc}.
A virtual link diagram is \s{almost classical} if it admits an Alexander numbering. A virtual link $L$ is \s{almost classical} if there is an almost classical virtual link diagram of $L$. If a virtual link diagram is almost classical, it is checkerboard colorable.

\subsection{cut system}

Let $D$ be a virtual link diagram. \s{An oriented cut system} or simply a \s{cut system} of $D$ is a set of oriented cut points on semi-arc of $D$ as depicted in Fig.\ref{cutpoint} such that $D$ with it admits an Alexander numbering, where numbers are given as in Fig.\ref{all}. Such an Alexander numbering is called an \s{Alexander numbering of a virtual link diagram with a cut system}. See Fig.\ref{ACvtrefoil} for example of a virtual link diagram with a cut system.

\begin{figure}[H]
\centering
\includegraphics[scale=0.4]{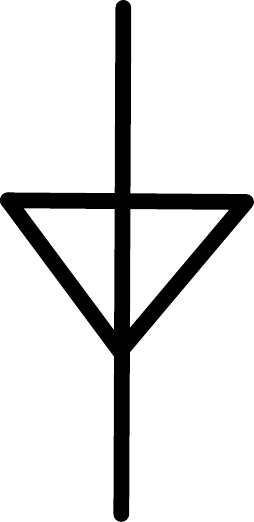}
\caption{An oriented cut point.} 
\label{cutpoint}
\end{figure}

\begin{figure}[H]
\centering
\includegraphics[scale=0.4]{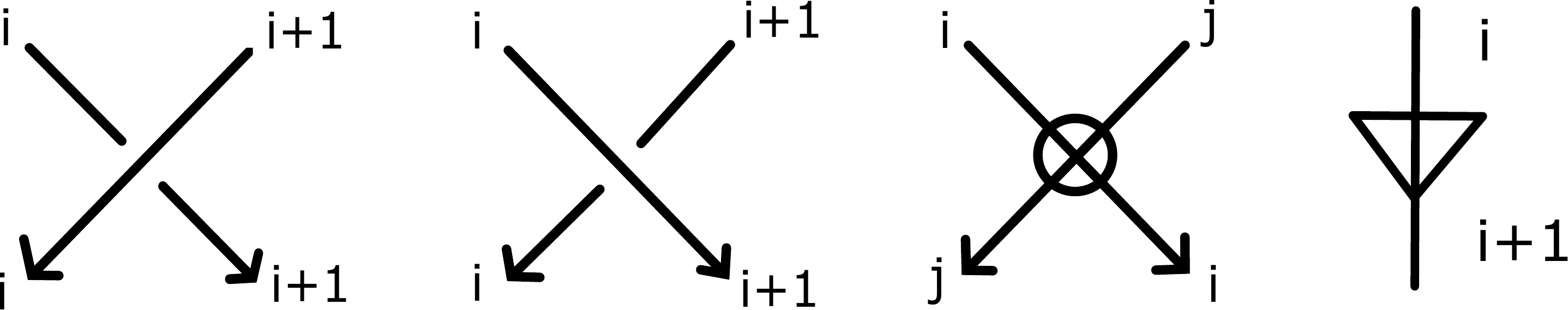}
\caption{Alexander numbering of a virtual link diagram with a cut system.} 
\label{all}
\end{figure}

\begin{figure}[H]
\centering
\includegraphics[scale=0.4]{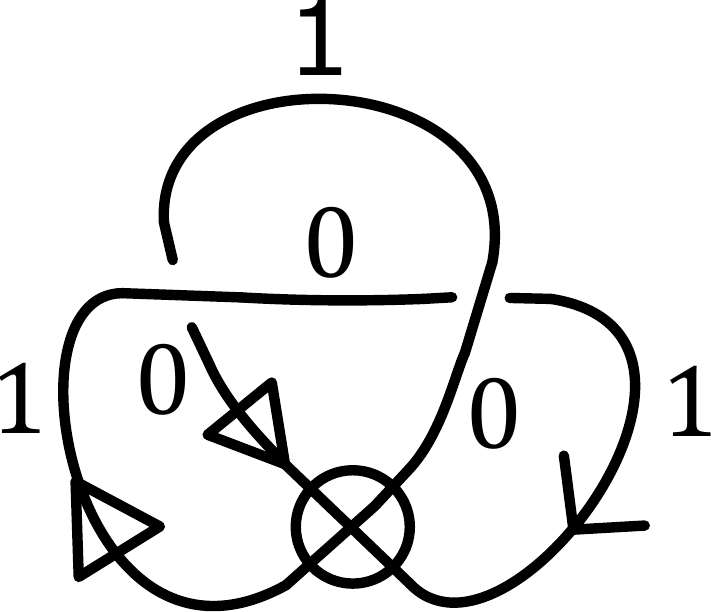}
\caption{Alexander numbering of a virtual link diagram with a cut system.} 
\label{ACvtrefoil}
\end{figure}

For a virtual link diagram $D$, its cut system is not unique. See an example of two cut systems of a virtual link diagram in Fig.\ref{8}.
\begin{figure}[H]
\centering
\includegraphics[scale=1.0]{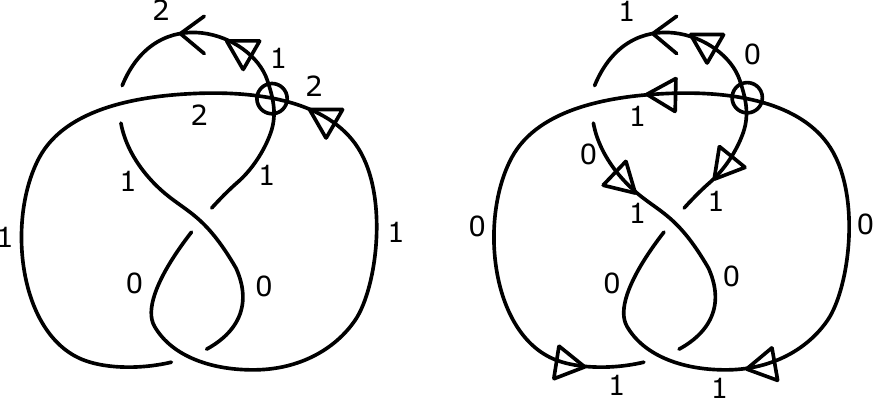}
\caption{Example of two cut systems of a virtual link diagram.} 
\label{8}
\end{figure}
The local transformations of oriented cut points depicted in Fig.\ref{cpmove} are called \s{oriented cut point moves}.  

\begin{figure}[htbp]
 \begin{minipage}{0.3\hsize}
   \centering
    \includegraphics[scale=0.28]{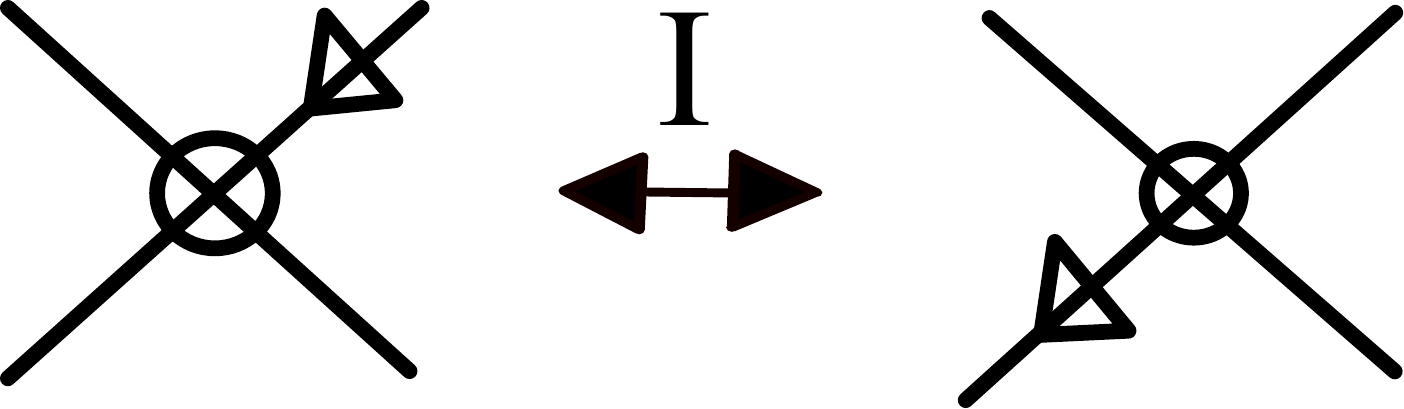}
 \end{minipage}
 \begin{minipage}{0.33\hsize}
   \centering
    \includegraphics[scale=0.28]{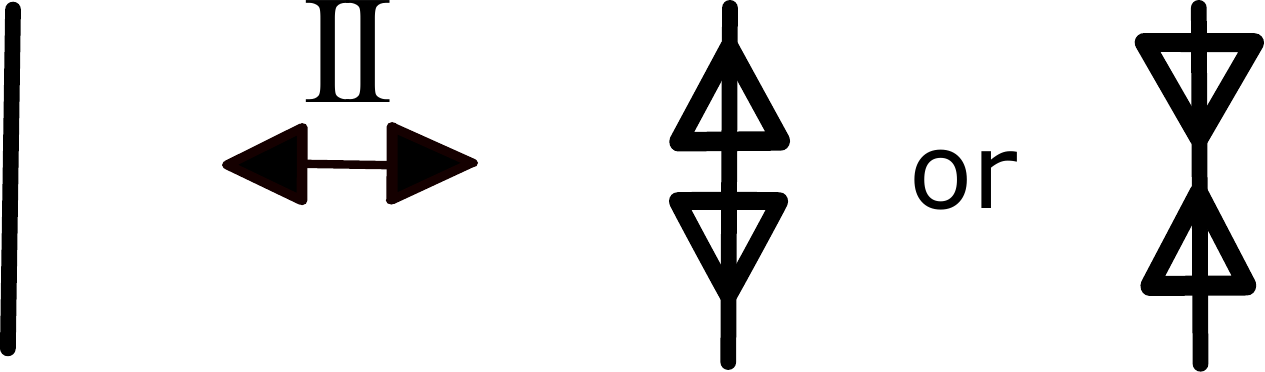}
 \end{minipage}
 \begin{minipage}{0.33\hsize}
   \centering
    \includegraphics[scale=0.28]{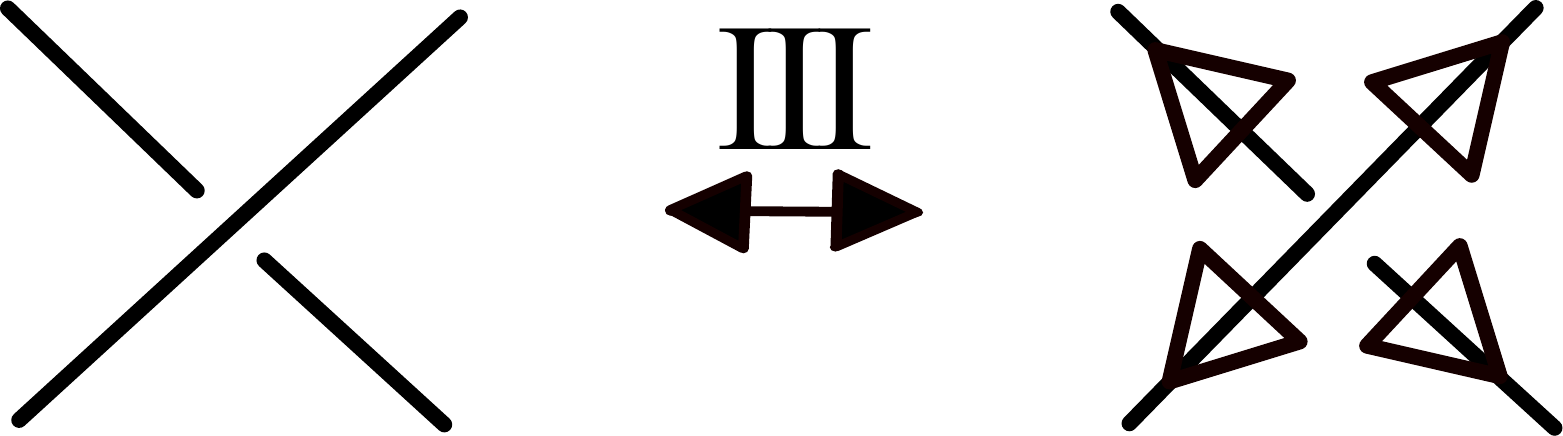}
 \end{minipage}
\caption{oriented cut point move.} 
\label{cpmove}
\end{figure} 

\begin{thm} (N. Kamada \cite{cp})\\
Two cut systems of a virtual limk diagram are related by a finite sequence of oriented cut point moves.
\end{thm}

Note that for an almost classical virtual link diagram $D$, the empty set is a cut system of $D$.

\subsection{A multivariable polynomial invariant}

A local replacement at a classical crossing of a virtual link diagram depicted in left of Fig.\ref{splice} (or right of Fig.\ref{splice}) is called A-\s{splice} (or B-\s{splice}). 

\begin{figure}[H]
\centering
\includegraphics[scale=0.4]{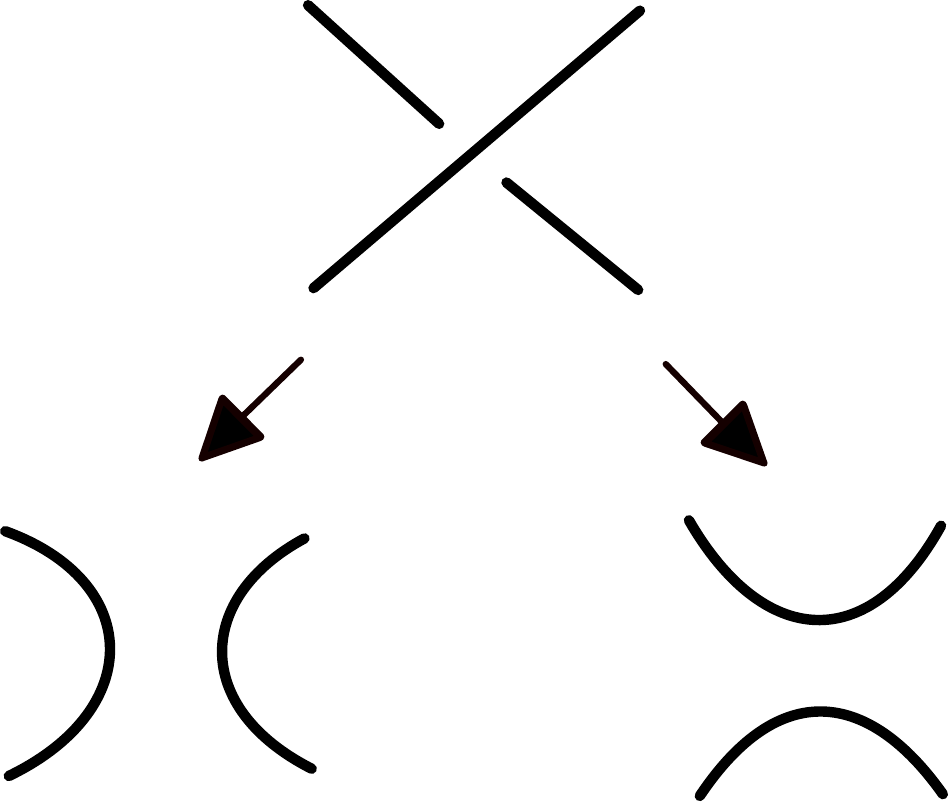}
\caption{splice.} 
\label{splice}
\end{figure}

Let ($D,C$) be a pair of a virtual link diagram $D$ and a cut system $C$. A \s{cut point state} (or \s{state}) of ($D,C$) is denoted as $S^c$ is a union of immersed loops in $\mathbb{R}^2$ with virtual crossings and cut points, which is obtained by splicing all classical crossings of $D$. 

We difine a map $\iota$ from the set of loops of cut point state diagrams to $\mathbb{Z}$ by the following conditions.

\begin{enumerate}
\item $\iota$ $\left( \begin{centering} \includegraphics[scale=0.2]{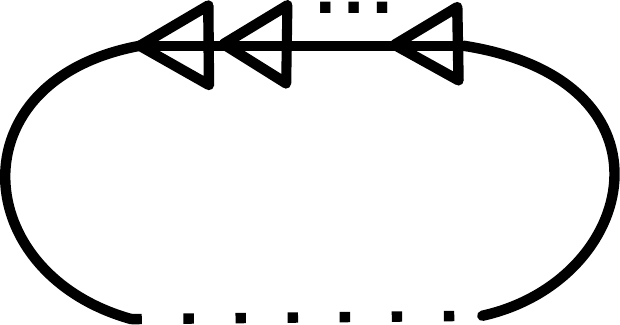} \end{centering} \right)$ = $\iota$ $\left( \begin{centering} \includegraphics[scale=0.2]{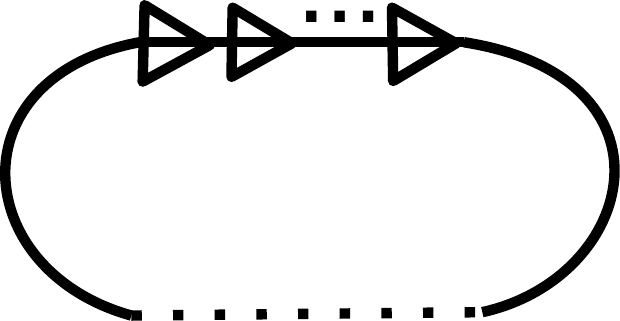} \end{centering} \right)$ = r , where 2r oriented cut points in the same direction appear.\\
\item $\iota$ $\left( \begin{centering} \includegraphics[scale=0.25]{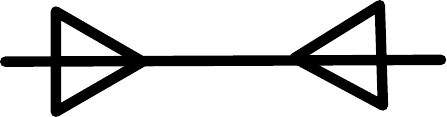} \end{centering} \right)$ = $\iota$ $\left( \begin{centering} \includegraphics[scale=0.25]{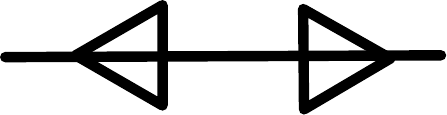} \end{centering} \right)$ = $\iota$ $\left( \begin{centering} \includegraphics[scale=0.25]{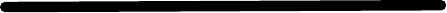} \end{centering} \right)$\\
\item $\iota$ $\left( \begin{centering} \includegraphics[scale=0.2]{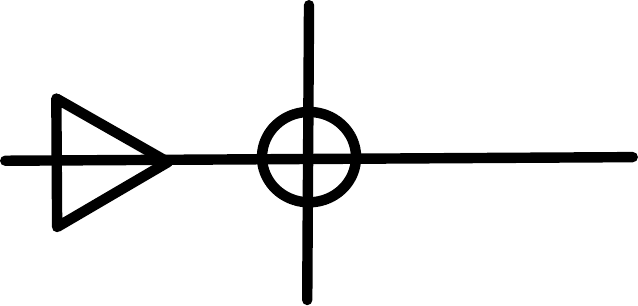} \end{centering} \right)$ = $\iota$ $\left( \begin{centering} \includegraphics[scale=0.2]{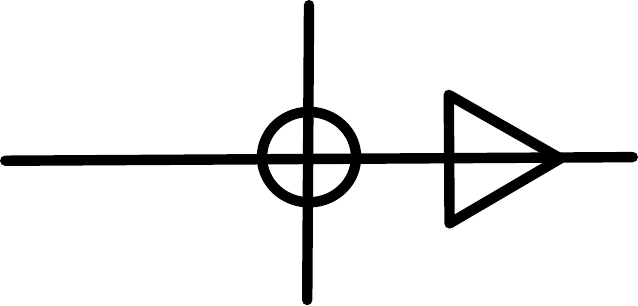} \end{centering} \right)$\\
\end{enumerate}

For a cut point state $S^c$ of ($D,C$), we denote $\natural$$S^c$ the number of A-splices minus B-splices obtaining $S^c$, by $\sharp$$S^c$ the number of loops in $S^c$, and by $\tau_i(S^c)$ the number of loops of $S^c$ whose indices by $\iota$ are i. 

The \s{double bracket} of ($D,C$) is defined by
\begin{equation*}
\braket{\hspace{-2pt}\braket{D,C}\hspace{-2pt}} \coloneqq  {\displaystyle \sum_{\sigma^c \in S^c}} A^{\natural \sigma^c} (-A^2-A^{-2})^{\sharp \sigma^c-1} d_1^{\tau_1(\sigma^c)}d_2^{\tau_2(\sigma^c)}\cdots \hspace{5pt} \in  \mathbb{Z} [A^{\pm 1},d_1,d_2 \cdots ].
\end{equation*}

\begin{prop}\label{choice} (N.Kamada \cite{jones})\\
Let $D$ be a virtual link diagram and $C$ and $C'$ be cut system of $D$. Then $\braket{\hspace{-2pt}\braket{D,C}\hspace{-2pt}}$ is equal to $\braket{\hspace{-2pt}\braket{D,C'}\hspace{-2pt}}$.
\end{prop}

We define $\braket{\hspace{-2pt}\braket{D}\hspace{-2pt}}$ by $\braket{\hspace{-2pt}\braket{D,C}\hspace{-2pt}}$ for a cut system $C$. Let a virtual link diagram $D$. The \s{writhe} of $D$ is denoted by $w(D)$, which is the number of positive crossing of $D$ minus the number of negative crossing of $D$. 

A \s{multivariable polynomial} of $D$ is defined by
\begin{equation*}
X_{D} \coloneqq (-A^3)^{-w(D)}\braket{\hspace{-2pt}\braket{D}\hspace{-2pt}}.
\end{equation*}

A multitvariable polynomial invariant for virtual links is defined by H. A. Dye, L. H. Kauffman (\cite{arrow}) and Y. Miyazawa (\cite{miyazawa}) independently.

\begin{thm} (N. Kamada \cite{jones})\\
Let $D$ be a virtual link diagram. The multivariable polynomial $X_D$ coincides with the multivariable polynomial invariant for virtual links  is defined by H. A. Dye, L. H. Kauffman, Y. Miyazawa.
\end{thm}

\section{Main result and applications}

\subsection{Main result}

\begin{thm}\label{syu}
  Let $(D_+,D_-,D_v)$ be a virtual skein triple. If $D_+,D_-$ are almost classical virtual link diagrams, then we have
\begin{equation*}
  (A^6-d_1)X_{D_+}+(-A^{-6}+d_1)X_{D_-}=(A^6-A^{-6})X_{D_v}.
\end{equation*}
\end{thm}

\begin{proof}

Let $p_+$ be a positive crossing of $D_+$ such that the corresponding crossing of $D_-$ (or $D_v$) is negative (or virtual) as in Fig.\ref{bunnkai}. The corresponding crossing of $D_-$ (or $D_+$) is denoted by $p_-$ (or $p_v$). We take empty sets as cut systems of $D_+$ and $D_-$. Then the virtual link diagram $D_v$ admits an Alexander numbering if two oriented cut points are given around the virtual crossing $p_v$ as in Fig.\ref{bunnkai}. We take such a cut system of $D_v$.

\begin{figure}[H]
\centering
\includegraphics[scale=0.5]{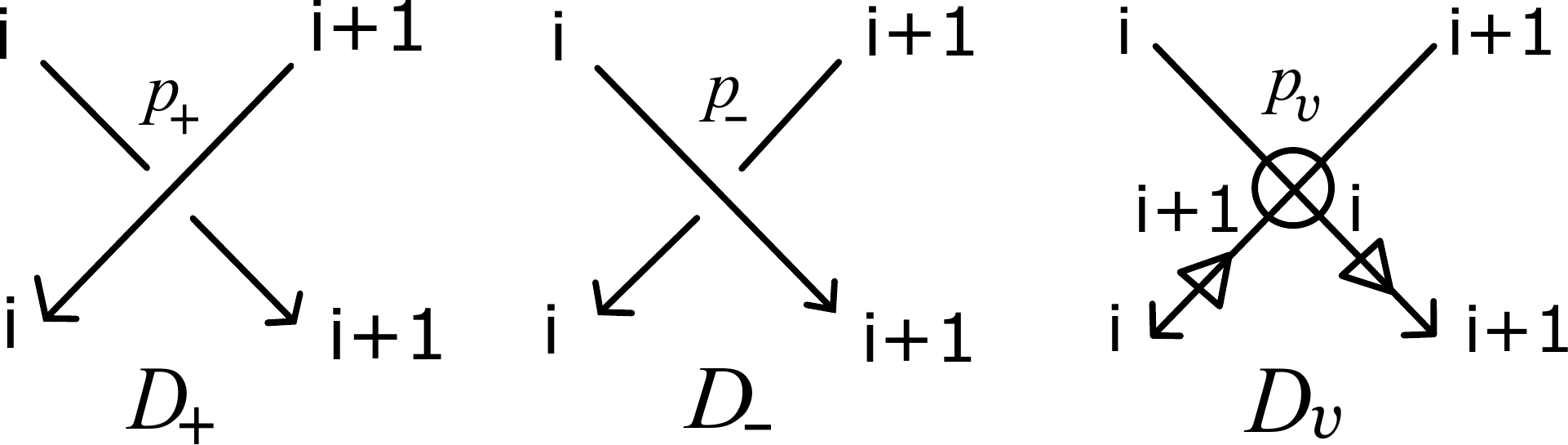}
\caption{Cut systems of virtual skein triple} 
\label{bunnkai}
\end{figure}

The loops of states of $D_+$ around a classical crossing $p_+$ are depicted as in Fig.\ref{dproof} (a) or (b) by noting that an almost classical virtual link diagram is checkerboard colorable. For the details, see the proof of Theorem 6 and Figure 9 in \cite{cc}. Those of $D_-$ (or $D_v$) around the corresponding classical crossing $p_-$ (or corresponding virtual crossing $p_v$) are depicted as in Fig.\ref{dproof} (c) or (d) (or Fig.\ref{dproof} (e) or (f)). The set of states of $D_+$ as depicted in Fig.\ref{dproof} (a) (or in Fig.\ref{dproof} (b)) is denoted by $S'_+$ (or $S''_+$). That of $D_-$ as depicted in Fig.\ref{dproof} (c)  (or in Fig.\ref{dproof} (d)) is denoted by $S'_-$ (or $S''_-$).  The set of states of $D_v$ as depicted in Fig. \ref{dproof}(e) (or Fig. \ref{dproof}(f)) is denoted by $S'_v$ (or $S''_v$).

\begin{figure}[htbp]
\centering
\includegraphics[scale=0.7]{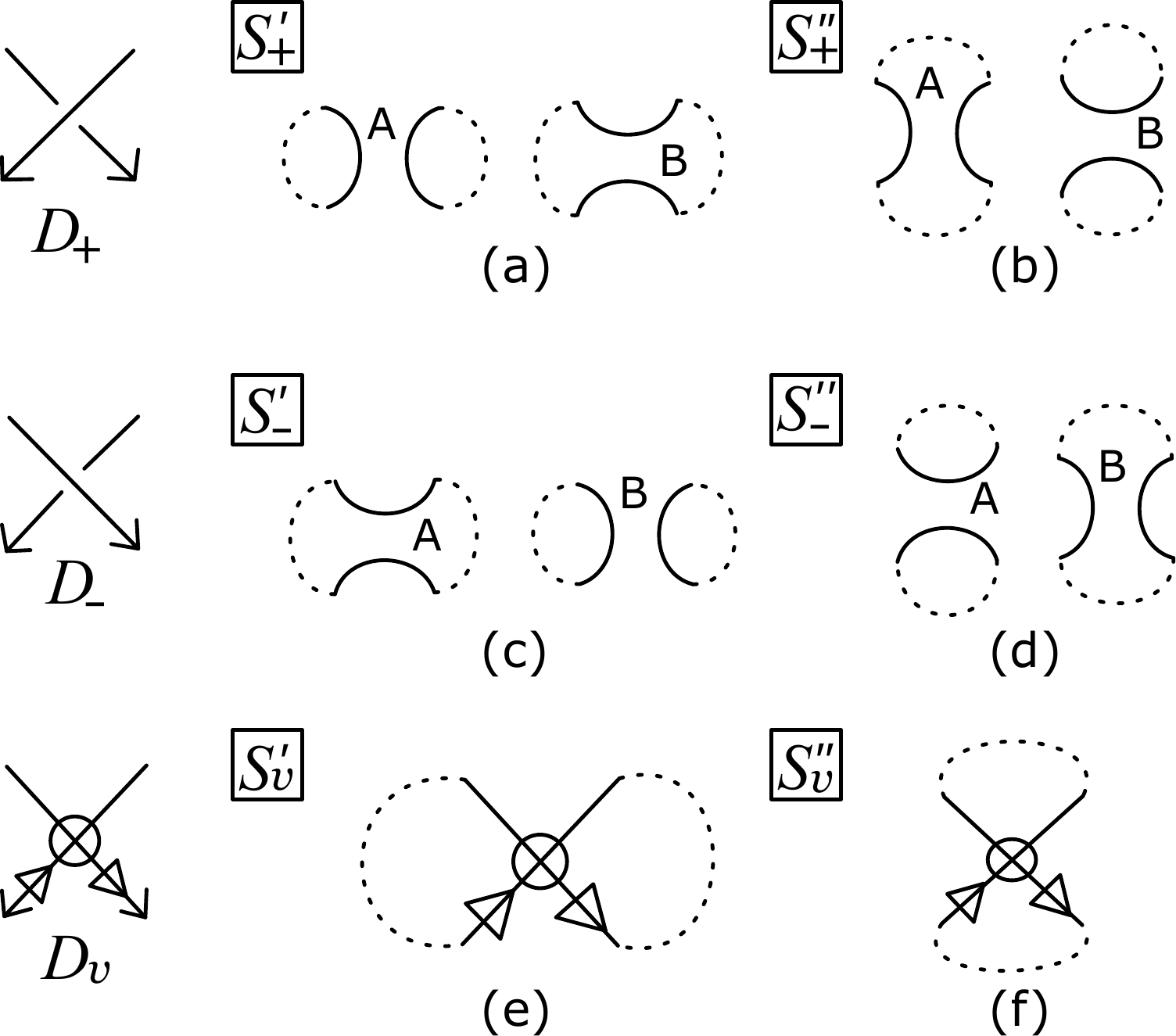}
\caption{Loops of states.}
\label{dproof}
\end{figure}

For a state $S_+$ of $S'_+$, let $S_v$ be the state of $S'_v$ corresponding to $S_+$. If $S_+$ is obtained from $D_+$ by applying A-splice (or B-splice) at the crossing $p_+$, $\sharp S_+$ = $\sharp S_v +1$ (or $\sharp S_+$ = $\sharp S_v$).

Then we have
\begin{align*}
\braket{\hspace{-2pt}\braket{D_+|S'_+}\hspace{-2pt}}
&= A(-A^2-A^{-2})\braket{\hspace{-2pt}\braket{D_v|S'_v}\hspace{-2pt}}+A^{-1}\braket{\hspace{-2pt}\braket{D_v|S'_v}\hspace{-2pt}} \notag \\
&= -A^3\braket{\hspace{-2pt}\braket{D_v|S'_v}\hspace{-2pt}}\\
\end{align*}
where, 
\begin{equation*}
\braket{\hspace{-2pt}\braket{D|S_+'}\hspace{-2pt}}  \coloneqq  {\displaystyle \sum_{\sigma^c \in S'_+}} A^{\natural \sigma^c} (-A^2-A^{-2})^{\sharp \sigma^c-1} d_1^{\tau_1(\sigma^c)}d_2^{\tau_2(\sigma^c)}\cdots \hspace{5pt} \in  \mathbb{Z} [A^{\pm 1},d_1,d_2 \cdots ].
\end{equation*}

We have the following in similar way.
\begin{align*}
\braket{\hspace{-2pt}\braket{D_-|S'_-}\hspace{-2pt}} 
&= -A^{-3}\braket{\hspace{-2pt}\braket{D_v|S'_v}\hspace{-2pt}}\\
d_1\braket{\hspace{-2pt}\braket{D_+|S''_+}\hspace{-2pt}}
&= (A+A^{-1}(-A^2-A^{-2}))\braket{\hspace{-2pt}\braket{D_v|S''_v}\hspace{-2pt}}\\
&= -A^{-3}\braket{\hspace{-2pt}\braket{D_v|S''_v}\hspace{-2pt}}\\ 
d_1\braket{\hspace{-2pt}\braket{D_-|S''_-}\hspace{-2pt}} 
&= -A^3\braket{\hspace{-2pt}\braket{D_v|S''_v}\hspace{-2pt}}.\\
\end{align*}

From the above
\begin{align*}
&(-A^3+A^{-3}d_1)\braket{\hspace{-2pt}\braket{D_+}\hspace{-2pt}}-(-A^{-3}+A^3d_1)\braket{\hspace{-2pt}\braket{D_-}\hspace{-2pt}}\\
&=(-A^3+A^{-3}d_1)(\braket{\hspace{-2pt}\braket{D_+|S'_+}\hspace{-2pt}}+\braket{\hspace{-2pt}\braket{D_+|S''_+}\hspace{-2pt}})-(-A^{-3}+A^3d_1)(\braket{\hspace{-2pt}\braket{D_-|S'_-}\hspace{-2pt}}+\braket{\hspace{-2pt}\braket{D_-|S''_-}\hspace{-2pt}})\\
&=(-A^3+A^{-3}d_1)(-A^3\braket{\hspace{-2pt}\braket{D_v|S'_v}\hspace{-2pt}}-A^{-3}d_1^{-1}\braket{\hspace{-2pt}\braket{D_v|S''_v}\hspace{-2pt}})-(-A^{-3}+A^3d_1)(-A^{-3}\braket{\hspace{-2pt}\braket{D_v|S'_v}\hspace{-2pt}}-A^3d_1^{-1}\braket{\hspace{-2pt}\braket{D_v|S''_v}\hspace{-2pt}})\\
&=(A^6-A^{-6})(\braket{\hspace{-2pt}\braket{D_v|S'_v}\hspace{-2pt}}+\braket{\hspace{-2pt}\braket{D_v|S''_v}\hspace{-2pt}})\\
&=(A^6-A^{-6})\braket{\hspace{-2pt}\braket{D_v}\hspace{-2pt}}. 
\end{align*}

Since $w(D_+)=w(D_v)+1$, $w(D_-)=w(D_v)-1$, it followd that
\begin{equation*}
(A^6-d_1)X_{D_+}+(-A^{-6}+d_1)X_{D_-}=(A^6-A^{-6})X_{D_v}. 
\end{equation*}
\end{proof}
\vspace{10pt}
The following is an example of Theorem \ref{syu}.
\begin{figure}[H]
\centering
\includegraphics[scale=0.7]{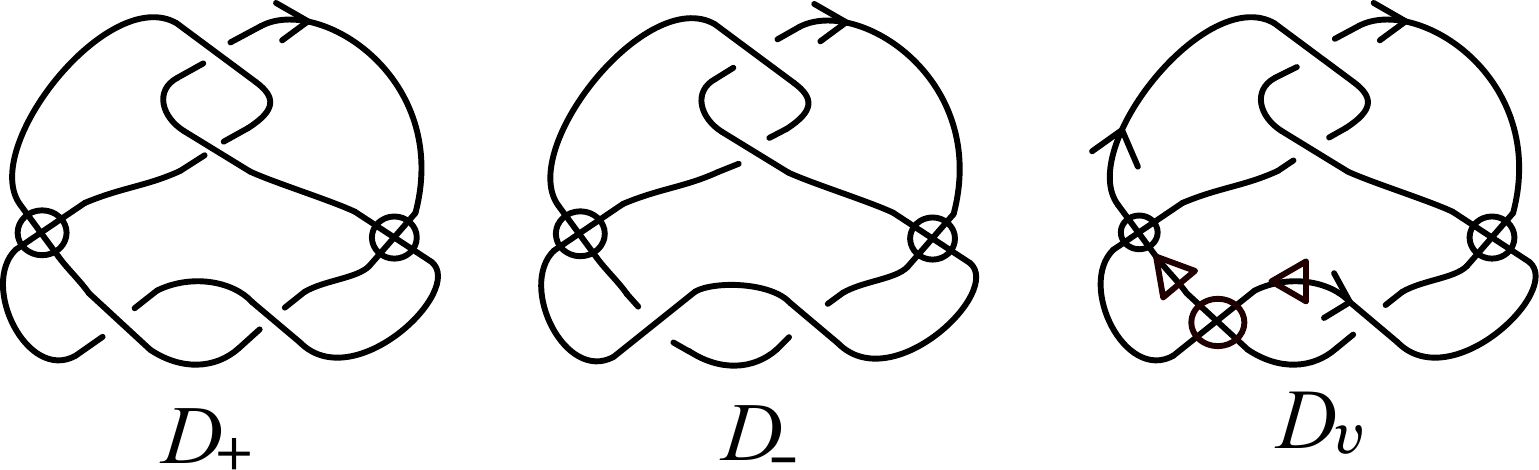}
\caption{Example.}
\label{thmex}
\end{figure}
The multivariable polynomials for each diagrams are as follows.
\begin{align*}
&X_{D_{+}}=A^8-A^4+1-A^{-4}+A^{-8}\\
&X_{D_{-}}=1 \\
&X_{D_{v}}=A^8-A^4+1+(-A^2+A^{-2})d_1\\
\end{align*}
So we have the following equations.
\begin{eqnarray*}
(A^6-d_1)X_{D_+}+(-A^{-6}+d_1)X_{D_-} &=& (A^6-d_1)(A^8-A^4+1-A^{-4}+A^{-8}) +(-A^{-6}+d_1)1 \\
              &=& A^{14}-A^{10}+A^6-A^2+A^{-2}-A^{-6}+(-A^8+A^4+A^{-4}-A^{-8})d_1  \\
(A^6-A^{-6})X_{D_v} &=& (A^6-A^{-6})(A^8-A^4+1+(-A^2+A^{-2}) d_1) \\
         &=& A^{14}-A^{10}+A^6-A^2+A^{-2}-A^{-6}+(-A^8+A^4+A^{-4}-A^{-8})d_1
\end{eqnarray*}

\subsection{Applications}

\begin{prop}\label{1} (T. Nakamura, Y. Nakanishi, S. Satoh, Y. Tomiyama \cite{almost},  N. Kamada \cite{jones})\\
Let $D$ be a virtual link diagram presenting an almost classical virtual link. Then, $X_D$ $\in  \mathbb{Z} [A^{\pm 1}]$. 
\end{prop}
By Theorem \ref{syu} and Proposition \ref{1}, we have the following.

\begin{cor}\label{coro}
If a virtual link diagram $D$ is obtained from an almost classical virtual link diagram by replacing a classical with a virtual crossing. Then, 
$X_{D} \in  \mathbb{Z} [A^{\pm 1},d_1].$
\end{cor}

For a polynomial $g \in  \mathbb{Z} [A^{\pm 1},d_1,d_2 \cdots]$, 
Exp($g$) is the set of integers appearing as exponents of $A$ in the term without $d_i$ in $g$.
Exp($g|d_i)$ is the set of integers appearing as exponent of $A$ in the term with $d_i$ in $g$ (i $\in$ $\mathbb{Z}$).

\begin{thm}\label{int}(N. Kamada\cite{cc})\\
Let $D$ be a checkerboard colorable $n$ component virtual link diagram. Then we have the following fr the $f$-polynomial of $D$, $f_D(A)$. Then\\
\begin{equation*}
\rm{Exp}(\mathit{f_D(A)}) \subset
\left\{ \begin{array}{ll}
4 \mathbb{Z}       \hspace{29pt}(n:odd)\\
4 \mathbb{Z} + 2  \hspace{10pt} (n:even)
\end{array}
\right. 
\end{equation*}
\end{thm}

\begin{cor}\label{cort}(c.f. S. Satoh, Y. Tomiyama \cite{prop})\\
Let $D$ be an $n$ component virtual link diagram obtained from an almost classical virtual link diagram by virtualizing a classical crossing. Then we have
\begin{align*}
& \rm{Exp}(\mathit{X_D}) \subset
\left\{ \begin{array}{ll}
4 \mathbb{Z}       \hspace{29pt}(n:odd)\\
4 \mathbb{Z} + 2  \hspace{10pt} (n:even)
\end{array}
\right. \\
& \rm{Exp}(\mathit{X_D|d_1}) \subset
\left\{ \begin{array}{ll}
4 \mathbb{Z} + 2 \hspace{10pt}(n:odd)\\
4 \mathbb{Z}      \hspace{29pt} (n:even)
\end{array}
\right.
\end{align*}
\end{cor}

\begin{proof}
Let $D$ be an $n$ component virtual link diagram obtained from an almost classical virtual link diagram $D_+$ by virtualizing a positive classical crossing $p$. Suppose that $D_-$ is an almost virtual link diagram obtained from $D_+$ by replacing a positive crossing $p$ with a negative crossing. Then we have 
\begin{equation*}
(A^6-d_1)X_{D_+}+(-A^{-6}+d_1)X_{D_-}=(A^6-A^{-6})X_{D_v}
\end{equation*}
from Theorem \ref{syu}.
By substituting 1 for $d_i$, the multivariable polynomial invariant coincide with $f$-polynomial.
From Proposition \ref{1}, $X_{D_+} \in \mathbb{Z} [A^{\pm 1}]$ and $X_{D_-} \in \mathbb{Z} [A^{\pm 1}]$. Therefore $X_{D_+}$ and $X_{D_-}$ coincide with $f$-polynomial of $D_+$ and $D_-$.
Since $D_+$ and $D_-$ are almost classical, then they are checkerboard colorable.
Then we have 
\begin{equation*}
\rm{Exp}(\mathit{X_{D_+}}) \subset
\left\{ \begin{array}{ll}
4 \mathbb{Z}       \hspace{29pt}(n:odd)\\
4 \mathbb{Z} + 2  \hspace{10pt} (n:even)
\end{array}
\right. 
\end{equation*}
from Theorem \ref{int}. Thus we have the result.
\end{proof}

\section*{Acknowledgement}

I would like to express my appreciation to Professer Naoko Kamada.

\end{document}